\documentclass[preprint,12pt]{elsarticle}

\usepackage{amssymb}
\usepackage{amsmath}
\usepackage{amsthm}

\newtheorem{theorem}{Theorem}[section]
\newtheorem{lemma}[theorem]{Lemma}
\newtheorem{corollary}[theorem]{Corollary}
\newtheorem{proposition}[theorem]{Proposition}

\theoremstyle{definition}

\theoremstyle{remark}
\newtheorem{remark}[theorem]{Remark}
\newtheorem{example}[theorem]{Example}
\newtheorem*{acknowledgements}{Acknowledgments}

\newcommand{\norm}[1]{\lVert#1\rVert}
\newcommand{\ran}{\mathrm{ran}\,}
\newcommand{\dom}{\mathrm{dom}\,}
\newcommand{\codim}{\mathrm{codim}\,}
\newcommand{\card}{\mathrm{card}\,}
\newcommand{\mul}{\mathrm{mul}\,}

\newcommand{\spc}{\mathrm{Sp}\,}

\begin{document}

\begin{frontmatter}



\title{Factorizations of Linear Relations}


\author[popovici]{Dan Popovici\corref{cor1}}
\ead{popovici@math.uvt.ro}
\cortext[cor1]{Corresponding author.}
\address[popovici]{Department of Mathematics, University of the West Timi\c soara, Bd. Vasile P\^arvan no.4, 300223, Timi\c soara, Romania}

\author[sebestyen]{Zolt\'an Sebesty\'en}
\ead{sebesty@cs.elte.hu}
\address[sebestyen]{Department of Applied Analysis and Computational Mathematics, E\"otv\"os Lor\'and University of Budapest, P\'azm\'any s\'et\'any 1/C, Budapest, H-1117, Hungary}

\begin{abstract}
Given two linear relations $A$ and $B$ we characterize the existence of a linear relation (operator) $C$ such that $A\subseteq BC$, respectively $A\subseteq CB.$ These factorizations extend and improve well-known results by R.G. Douglas and Z. Sebesty\'en.
\end{abstract}

\begin{keyword}
Linear relation; unbounded operator; Douglas factorization theorem

\end{keyword}

\end{frontmatter}


\section{Introduction}
\label{S1}

R.G. Douglas proposed in \cite{Do} two conditions on a given pair (A, B) of bounded linear operators acting on a Hilbert space $\mathfrak{H}$ which are equivalent to the existence of a bounded linear operator $C$ on $\mathfrak{H}$ such that the factorization $A=BC$ holds true. The first one, which was actually indicated by P. Halmos, relates the ranges of $A$ and $B$ (more exactly, $\ran A\subseteq\ran B$), while the second one is a majorization result between the positive operators $AA^*$ and $BB^*$ (more exactly, $AA^*\le \lambda BB^*$ for some $\lambda\ge 0$). The solution $C$ to the operator equation $A=BC$ is uniquely determined if we require, in addition, that $\ran C\subseteq \overline{\ran B^*}$. This particular solution, called the reduced (or Douglas) solution satisfies, in addition, the conditions $\ker C=\ker A$ and $\norm C^2=\inf\{\mu\ge 0 : AA^*\le\mu BB^*\}.$ One common approach into the study of the Douglas equation $AX=B$ involves the theory of generalized inverses. In fact, as it was shown in \cite{EN} (cf. also \cite{ACG1}), the reduced solution $C$ of the above equation can be computed explicitly in terms of the Moore-Penrose inverse $B^\dagger$ of $B$, more precisely as $C=B^\dagger A.$ Other facts and applications relating the Douglas theorem and the theory of generalized inverses can be found in \cite{Na}. This theorem has been recently used as an important tool into the study of $A$-operators (with $A$ positive) \cite{ACG2, ACG3}, the invertibility of operator matrices \cite{HC}, or the theory of commuting operator tuples associated with the unit ball in $\mathbb{C}^d$ \cite{RS}. Generalized versions have been proposed by M.R. Embry in \cite{Em} (for operators on Banach spaces) and by X. Fang, J. Yu and H. Yao in \cite{FYY} (for adjointable operators on Hilbert $C^*$-modules).

As a potential tool in the theory of linear partial differential equations Douglas formulated the theorem above for the case when the operators $A$ and $B$ are only closed and densely defined. He showed that the range inclusion $\ran A\subseteq \ran B$ is sufficient for a factorization of the form $A\subseteq BC$, where $C$ is a certain densely defined operator. Z. Sebesty\'en \cite{Se} improved the results of Douglas and factorized a densely defined operator $A$ as $A\subseteq BC,$ where $B$ is the adjoint of a densely defined operator, $\ran A\subseteq\ran B$ and $C$ is minimal in the sense that
\begin{equation*}
\norm {Cx} \le \norm y \mbox{ for } x\in\dom A\mbox{ and }y\in\dom B\mbox{ such that }Ax=By.
\end{equation*}

The concept of linear relation between linear spaces has been introduced by R. Arens \cite{Ar} in order to extend results in operator theory from singlevalued to multivalued case. This notion had theoretical implications in various domains and it was used in several applications \cite{Cr}. These facts motivated us to study the Douglas theorem for the generalized framework of linear relations. To be more precise, our main goal in this paper is that, for two given linear relations $A$ and $B$ to characterize the existence of a linear relation $C$ such that $A\subseteq BC$, respectively $A\subseteq CB.$ The case when $C$ is required to be  (the graph of) an operator is also described. 

The paper is organized as follows. In Section 2 we prove that the range inclusion $\ran A\subseteq\ran B$ is necessary and sufficient for the existence of a linear relation $C$ for which $A\subseteq BC.$ A more precise solution is $C=B^{-1}A.$ As a consequence, the problem concerning the existence of a linear relation $C$ such that $A\subseteq CB$ is also solved. Section 3 contains information regarding the existence of an operator $C$ such that the factorization $A\subseteq BC$ holds true. More exactly, we show that the problem $A\subseteq BX$ has an operator solution if and only if
\begin{equation*}
\ran A\subseteq\ran B\quad\text{ and }\quad\mul A\subseteq\mul B.
\end{equation*}
We extend, in particular, the theorems of R. G. Douglas \cite{Do} and Z. Sebesty\'en \cite{Se} mentioned above. The problem $A\subseteq XB$ with operator solutions $X=C$ is considered in the last two parts. Firstly, we associate an operator $C$ to any Hamel basis $\{ z_\alpha\}_{\alpha\in I}$  for the range of $A$, any Hamel basis $\{ x'_\beta\}_{\beta\in J}$ for the kernel of $A$, any linearly independent family $\{ y_\alpha\}_{\alpha\in I}$ with $y_\alpha\in BA^{-1}(z_\alpha),\ \alpha\in I$ and any family $\{ y'_\beta\}_{\beta\in J}$ with $y'_\beta\in B(x'_\beta),\ \beta\in J$ such that the subspaces generated by $\{y_\alpha\}_{\alpha\in I}$ and, respectively, by $\{y'_\beta\}_{\beta\in J}$ have null intersection. We identify, in this context, the general form of all solutions. Secondly, we prove that the proposed problem has an operator solution if and only if 
\begin{equation*}
\dom A\subseteq\dom B\quad\text{ and }\quad\dim [A(\dom A\cap\ker B)]\le \dim (\mul B),
\end{equation*}
where $\dom, \ker$ and $\mul$ designate the domain, kernel, respectively the multivalued part of a given linear relation. Several applications are also included.

\section{Douglas-Type Problems for Linear Relations}\label{s2}

Throughout the rest of the paper the symbols $\mathfrak{X}, \mathfrak{Y}$ and $\mathcal{Z}$ denote linear spaces over the real or complex field $\mathbb{K}$. A linear relation (multivalued operator) between $\mathfrak{X}$ and $\mathfrak{Y}$ is a linear subspace $R$ of the cartesian product $\mathfrak{X}\times \mathfrak{Y}$. The inverse 
\begin{equation*}
R^{-1}:=\{ (y, x) : (x, y)\in R\}
\end{equation*}
of $R$ is a linear relation between $\mathfrak{Y}$ and $\mathfrak{X}$. 

If $\mathfrak{X}_0$ is a subset of $\mathfrak{X}$ then the image of $\mathfrak{X_0}$ is defined as
\begin{equation*}
R(\mathfrak{X}_0 ):=\{ y\in\mathfrak{Y} : (x, y)\in R \mbox{ for some } x\in\mathfrak{X}_0\}.
\end{equation*}
For simplicity we write, for a given $x\in\mathfrak{X}, R(x)$ instead of $R(\{x\})$. If $\mathfrak{X}_0$ is a linear subspace then $R(\mathfrak{X}_0)$ is also a linear subspace. In particular, the domain $\dom R:=R^{-1}(\mathfrak{Y})$ and kernel $\ker R:=R^{-1}(0)$ are linear subspaces of $\mathfrak{X}$, while the range $\ran R :=R(\mathfrak{X})$ and the multivalued part $\mul R:=R(0)$ are linear subspaces of $\mathfrak{Y}$. 

Given two linear relations $R\subseteq\mathfrak{X}\times\mathfrak{Y}$ and $S\subseteq\mathfrak{Y}\times\mathfrak{Z}$ the product $SR$ is a linear relation from $\mathfrak{X}$ into $\mathfrak{Z}$, defined by
\begin{equation*}
SR:=\{(x, z) : (x, y)\in R \mbox{ and } (y, z)\in S \mbox{ for some } y\in\mathfrak{Y}\}.
\end{equation*}
Easy computations show that
\begin{equation}\label{eq0}
\dom (SR)=R^{-1}(\dom S\cap \ran R),\quad \ker (SR)=R^{-1}(\ker S\cap \ran R)
\end{equation}
and
\begin{equation}\label{eq0p}
\ran (SR)=S(\dom S\cap \ran R),\quad \mul(SR)=S(\dom S\cap \mul R).
\end{equation}

R. Arens \cite{Ar} characterized the equality of two linear relations in terms of their kernels and ranges:
\begin{proposition}\label{p0}
Let $R$ and $S$ be two linear relations between $\mathfrak{X}$ and $\mathfrak{Y}$ such that $R\subseteq S.$ Then $R=S$ if and only if $\ker R=\ker S$ and $\ran R=\ran S.$
\end{proposition}

Our first result solves the factorization problem of R.G. Douglas in the generalized framework of linear relations.

\begin{theorem}\label{t1}
Let $A\subseteq\mathfrak{X}\times\mathfrak{Z}$ and $B\subseteq\mathfrak{Y}\times\mathfrak{Z}$ be two linear relations. The following statements are equivalent:
\begin{itemize}
\item[$(i)$] There exists a linear relation $C\subseteq\mathfrak{X}\times\mathfrak{Y}$ such that $A\subseteq BC$;
\item[$(ii)$] $\ran A\subseteq \ran B$;
\item[$(iii)$] $A\subseteq BB^{-1}A$.
\end{itemize}
\end{theorem}

\begin{proof}
The implications $(i)\Rightarrow(ii)$ and $(iii)\Rightarrow(i)$ are obvious.

$(ii)\Rightarrow(iii).$ Let $(x, z)\in A$. Note that, according to the definition of the product relation $BB^{-1}A$,
\begin{multline}\label{eq1}
(x, z)\in BB^{-1}A \mbox{ if and only if }\\
(x, u)\in A \mbox{ and } (y, u), (y, z)\in B \mbox{ for certain } y\in\mathfrak{Y} \mbox{ and } u\in\mathfrak{Z}.
\end{multline}
Since $z\in\ran B$ ($\supseteq \ran A$), there exists $y\in\mathfrak{Y}$ such that $(y, z)\in B.$ We can take $u=z$ to obtain that $(x, u)=(x, z)\in A$ and $(y, u)=(y, z)\in B.$ Hence $(x, z)\in BB^{-1}A$, as required.
\end{proof}

\begin{corollary}\label{c2}
Let $A\subseteq\mathfrak{X}\times\mathfrak{Z}$ and $B\subseteq\mathfrak{X}\times\mathfrak{Y}$ be two linear relations. The following statements are equivalent:
\begin{itemize}
\item[$(i)$] There exists a linear relation $C\subseteq\mathfrak{Y}\times\mathfrak{Z}$ such that $A\subseteq CB$;
\item[$(ii)$] $\dom A\subseteq \dom B$;
\item[$(iii)$] $A\subseteq AB^{-1}B$.
\end{itemize}
\end{corollary}

\begin{proof}
In view of the formulas $\dom A=\ran(A^{-1}),\ \dom B=\ran(B^{-1})$ and
\begin{equation*}
A^{-1}\subseteq B^{-1}BA^{-1}\quad \mbox{iff}\quad A\subseteq(B^{-1}BA^{-1})^{-1}=AB^{-1}B
\end{equation*}
the proof follows immediately by Theorem \ref{t1} for the linear relations $A^{-1}$ and $B^{-1}$.
\end{proof}

Let us now suppose that $A$ and $B$ are linear relations between $\mathfrak{X}$ and $\mathfrak{Z}$, respectively $\mathfrak{Y}$ and $\mathfrak{Z}$. We observe, by \eqref{eq1}, that $z\in\ran(BB^{-1}A)$ if and only if
\begin{equation}\label{eq2}
(x, u)\in A \mbox{ and } (y, u), (y, z)\in B \mbox{ for certain } x\in\mathfrak{X}, y\in\mathfrak{Y} \mbox{ and } u\in\mathfrak{Z}.
\end{equation}
In other words, there exist $x\in\mathfrak{X}$ and $u\in\mathfrak{Z}$ such that $(x, u)\in A$ and $(0, u-z)\in B.$ More exactly, \eqref{eq2} can be rewritten in equivalent form as
\begin{equation*}
u-z\in\mul B \mbox{ for a certain } u\in\ran A.
\end{equation*}
We proved that
\begin{equation}\label{eq3}
\ran (BB^{-1}A)=\ran A+\mul B.
\end{equation}

Similarly, $x\in\ker(BB^{-1}A)$ if and only if
\begin{equation*}
(x, u)\in A \mbox{ and } (y, u)\in B \mbox{ for certain } u\in\mathfrak{Z} \mbox{ and } y\in\ker B.
\end{equation*}

Equivalently, there exists $u\in\mul B$ such that $(x, u)\in A.$ We deduce that $x\in A^{-1}(\mul B).$ Consequently,
\begin{equation}\label{eq4}
\ker(BB^{-1}A)=A^{-1}(\mul B).
\end{equation}

We take $R=A$ and $S=BB^{-1}A$ in Proposition \ref{p0} to obtain, according to formulas \eqref{eq3} and \eqref{eq4} and to Theorem \ref{t1}, that:

\begin{corollary}\label{c3}
Let $A\subseteq\mathfrak{X}\times\mathfrak{Z}$ and $B\subseteq\mathfrak{Y}\times\mathfrak{Z}$ be two linear relations. The following statements are equivalent: 
\begin{itemize}
\item[$(i)$]$A=BB^{-1}A$;
\item[$(ii)$] $\mul B\subseteq \ran A\subseteq \ran B$ and $A^{-1}(\mul B)\subseteq\ker A.$ 
\end{itemize}
In particular, if $B$ is (the graph of) an operator and $\ran A\subseteq \ran B$ then there exists a linear relation $C\subseteq\mathfrak{X}\times\mathfrak{Y}$ (one possible solution is $C=B^{-1}A)$ such that $A=BC.$ 
\end{corollary}

\begin{corollary}\label{c4}
Let $A\subseteq\mathfrak{X}\times\mathfrak{Z}$ and $B\subseteq\mathfrak{C}\times\mathfrak{Y}$ be two linear relations. The following statements are equivalent:
\begin{itemize}
\item[$(i)$] $A=AB^{-1}B$;
\item[$(ii)$] $\ker B\subseteq\dom A\subseteq \dom B$ and $A(\ker B)\subseteq\mul A.$
\end{itemize}
In particular, if $B^{-1}$ is (the graph of) an operator and $\dom A\subseteq \dom B$ then there exists a linear relation $C\subseteq\mathfrak{Y}\times\mathfrak{Z}$ (one possible solution is $C=AB^{-1}$) such that $A=CB.$
\end{corollary}

\section{The Problem $A\subseteq BX$ with Operator Solutions}\label{s3}

It is our aim in this section to characterize, for two given linear relations $A$ and $B$, the existence of an operator $C$ such that $A\subseteq BC.$

Let us firstly note that, by the last part of \eqref{eq0p} (specialized for $S=B^{-1}$ and $R=A$) and according to Theorem \ref{t1}, the conditions $\ran A\subseteq \ran B$ and $B^{-1}(\mul A)=\{0\}$ are sufficient for the existence of an operator $C$ such that $A\subseteq BC.$

The proposed problem can be completely solved by the following:
\begin{theorem}\label{t5}
Let $A\subseteq\mathfrak{X}\times\mathfrak{Z}$ and $B\subseteq\mathfrak{Y}\times\mathfrak{Z}$ be two linear relations. The following statements are equivalent:
\begin{itemize}
\item[$(i)$] There exists an operator $C$ between $\mathfrak{X}$ and $\mathfrak{Y}$ such that that $A\subseteq BC;$
\item[$(ii)$] $\ran A\subseteq \ran B$ and $\mul A\subseteq \mul B;$
\item[$(iii)$] For every $x\in\dom A$ there exists $y\in\dom B$ such that $A(x)\subseteq B(y).$
\end{itemize}
\end{theorem}

\begin{proof}
$(i)\Rightarrow (ii).$ If $A\subseteq BC$ then, obviously, $\ran A\subseteq \ran B.$ In addition, if $z\in \mul A$ (i.e., $(0, z)\in A$) then $(y, z)\in B$ for a certain $y\in\mul C.$ Since $C$ is (the graph of) an operator (i.e., $\mul C=\{0\}$) we obtain that $y=0$ so $z\in\mul B.$ Hence $\mul A\subseteq \mul B.$

$(ii)\Rightarrow (iii).$ Let $x\in \dom A$ and $z\in \mathfrak{Z}$ such that $(x, z)\in A.$ As $z\in\ran A\subseteq\ran B$ it follows that $(y, z)\in B$ for a certain $y\in\mathfrak{Y}.$

We prove that $A(x)\subseteq B(y).$ To this aim let $z'\in A(x).$ Then $z'\in \ran A\subseteq\ran B$ so there exists $y'\in\mathfrak{Y}$ such that $(y', z')\in B$. Since $(x, z)$ and $(x, z')$ are both elements of $A$ it follows that $z-z'\in\mul A\subseteq\mul B.$ We deduce that $(y, z')=(y, z)-(0, z-z')\in B.$ Thus $z'\in B(y)$ as required.

$(iii)\Rightarrow(i).$ The proof follows by an application of Zorn's lemma.

Let us consider the set
\begin{multline*}
\mathcal{F}:=\bigl\{ (\mathfrak{A}, C) : \mathfrak{A} \mbox{ is a subspace of } \dom A,\ C \mbox{ is (the graph of)}\\ 
\mbox{ an operator with domain } \mathfrak{A} \mbox{ and } (\mathfrak{A}\times \mathfrak{Z})\cap A\subseteq BC\bigr\}
\end{multline*}
endowed with the partial order
\begin{multline*}
(\mathfrak{A}_1, C_1)\le (\mathfrak{A}_2, C_2)\quad\mbox{ if }\quad \mathfrak{A}_1\subseteq\mathfrak{A}_2 \mbox{ and } C_1\subseteq C_2,\\
((\mathfrak{A}_1, C_1),(\mathfrak{A}_2, C_2)\in\mathcal{F}).
\end{multline*}

\medskip
$(a)$ $\mathcal{F}\ne \emptyset.$

\medskip
Let $x\in\dom A$ and $y\in\dom B$ such that $A(x)\subseteq B(y)$ (according to $(iii)$). Note that, if $x=0$ and $z\in A(0)\subseteq B(y)$, then $z\in B(0)$ (as $(y, z)$ and $(y, 0)$ are both elements of $B$). Consequently, if $x=0$ then we can safely consider $y=0.$ Let us define
\begin{equation*}
\mathfrak{A}_x:=\{ \lambda x\}_{\lambda\in\mathbb{K}}\quad\mbox{and}\quad C_x=\{ (\lambda x, \lambda y)\}_{\lambda\in\mathbb{K}}.
\end{equation*}
Then $\mathfrak{A}_x$ is a linear subspace of $\dom A$ and $C_x$ is (the graph of) an operator on $\dom C_x=\mathfrak{A}_x.$ Moreover, if $(x', z)\in (\mathfrak{A}_x\times \mathfrak{Z})\cap A$ then $x'=\lambda x$ for a certain $\lambda\in\mathbb{K}, (x', \lambda y)\in C_x$ and, since $z\in A(\lambda x)\subseteq B(\lambda y), (\lambda y, z)\in B.$ Hence $(x', z)\in BC_x.$ It follows that $(\mathfrak{A}_x\times\mathfrak{Z})\cap A\subseteq BC_x.$ Hence $(\mathfrak{A}_x, C_x)\in\mathcal{F}.$

\bigskip
$(b)$ \textit{If $(\mathfrak{A}, C)\in\mathcal{F}$ and $\mathfrak{A}\ne\dom A$ then there exists $(\tilde{\mathfrak{A}}, \tilde{C})\in\mathcal{F}$ such that $(\mathfrak{A}, C)\le (\tilde{\mathfrak{A}}, \tilde{C})$ and $(\mathfrak{A}, C)\ne(\tilde{\mathfrak{A}}, \tilde{C})$.}

\medskip
Let $x\in\dom A\setminus \mathfrak{A}$ and $y\in\dom B$ such that $A(x)\subseteq B(y).$ We define
\begin{equation*}
\tilde{\mathfrak{A}}:=\{ \lambda x\}_{\lambda\in\mathbb{K}}\oplus \mathfrak{A}\quad\text{and}\quad
\tilde{C}:=\{ (\lambda x+a, \lambda y+b) : \lambda\in\mathbb{K}, (a, b)\in C\}
\end{equation*}
(the symbol ``$\oplus$'' denotes a direct sum). It is not hard to observe that $\tilde{C}$ is (the graph) of an operator between $\mathfrak{X}$ and $\mathfrak{Y}$, it extends $C$ and $\dom\tilde{C}=\tilde{\mathfrak{A}}.$ In order to prove that $(\tilde{\mathfrak{A}}\times\mathfrak{Z})\cap A\subseteq B\tilde{C}$ we take $\lambda\in\mathbb{K}, a\in\mathfrak{A}$ and $z\in\mathfrak{Z}$ such that $(\lambda x+a, z)\in A.$ Let $z_2\in\mathfrak{Z}$ with $(a, z_2)\in (\mathfrak{A}\times\mathfrak{Z})\cap A\subseteq BC,$ so $(C(a), z_2)\in B.$ Also, for $z_1\in A(\lambda x)\subseteq B(\lambda y),$
\begin{equation*}
(0, z-z_1-z_2)=(\lambda x+a, z)-(\lambda x, z_1)-(a, z_2)\in A.
\end{equation*}
In other words $z-z_1-z_2\in A(0)\subseteq B(0).$ We deduce that 
\begin{equation*}
(\tilde{C}(\lambda x+a), z)=(\lambda y, z_1)+(C(a), z_2)+(0, z-z_1-z_2)\in B.
\end{equation*}
Therefore $(\tilde{\mathfrak{A}}\times\mathfrak{Z})\cap A\subseteq B\tilde{C}.$

\bigskip
$(c)$ \textit{Every chain $\mathcal{F}_0=\{ (\mathfrak{A}_\alpha, C_\alpha)\}_{\alpha\in I}$ has an upper bound in $\mathcal{F}.$}

\medskip
Indeed, the chain $\mathcal{F}_0$ has in $\mathcal{F}$ the upper bound $(\bigcup_{\alpha\in I} \mathfrak{A}_\alpha, \bigcup_{\alpha\in I} C_\alpha).$

\medskip
According to Zorn's lemma $\mathcal{F}$ has a maximal element $(\mathfrak{A}, C).$ Clearly $\mathfrak{A}=\dom A$ since, otherwise, $(\mathfrak{A}, C)$ can be ``strictly'' extended in $\mathcal{F}$ (by $(b)$), which contradicts its maximality. Also, by the definition of $\mathcal{F}, A=(\mathfrak{A}\times\mathfrak{Z})\cap A\subseteq BC.$ This completes the proof.
\end{proof}

Our result extends and improves the theorems of R.D. Douglas and Z. Sebesty\'en mentioned in the introduction:
\begin{corollary}\label{c6}
Let $A$ be (the graph of) an operator between $\mathfrak{X}$ and $\mathfrak{Z}$ and $B$ a linear relation between $\mathfrak{Y}$ and $\mathfrak{Z}.$ The following statements are equivalent: 
\begin{itemize}
\item[$(i)$] There exists an operator $C$ between $\mathfrak{X}$ and $\mathfrak{Y}$ such that $A\subseteq BC$;
\item[$(ii)$] $\ran A\subseteq \ran B.$
\end{itemize}
\end{corollary}

\begin{remark}\label{r7}
Let $A\subseteq \mathfrak{X}\times\mathfrak{Z}$ and $B\subseteq \mathfrak{Y}\times\mathfrak{Z}$ be two linear relations satisfying one (and hence all) of the equivalent statements of Theorem \ref{t5} and consider a particular operator solution $X=C_0$ of the problem $A\subseteq BX.$ A given operator $C$ between $\mathfrak{X}$ and $\mathfrak{Y}$ is a solution to this problem if and only if $C$ extends the operator 
\begin{equation*}
C_0|_{\dom A}+C_1
\end{equation*}
for a certain operator $C_1$ with domain $\dom A$ and range contained in $\ker B.$
\qed
\end{remark}

\section{The Problem $A\subseteq XB$ with Operator Solutions. Linearly Independent Systems}\label{s4}

We pass now to the problem regarding the existence, for two given linear relations $A\subseteq \mathfrak{X}\times\mathfrak{Z}$ and $B\subseteq \mathfrak{X}\times\mathfrak{Y}$, of an operator $C$ between $\mathfrak{Y}$ and $\mathfrak{Z}$ such that $A\subseteq CB.$

We start our discussion with a particular case:

\begin{remark}\label{r8}
\textit{If $\ran A=\{ 0\}$ then there exists and operator $C$ such that $A\subseteq CB$ if and only if $\dom A\subseteq \dom B$; one possible solution is the null operator with domain $\ran B.$}

The direct implication being obvious we only have to prove that if $\dom A$ $\subseteq \dom B$ and $C=0_{\ran B}$ then $A\subseteq CB.$ Indeed, if $x\in\ker A\subseteq \dom B$ then there exists $y\in \mathfrak{Y}$ with $(x, y)\in B.$ Since $C(y)=0$ it follows that $(x, 0)\in CB$, as required.
\qed
\end{remark}

If $\mathfrak{X}_0$ is a subset of $\mathfrak{X}$ we denote by $\spc(\mathfrak{X}_0)$ the linear subspace of $\mathfrak{X}$ generated by $\mathfrak{X}_0$. In the following $\{ 0\}$ will be conventionally considered as the only linearly independent system (Hamel basis) of the null space.

We are now in position to solve the proposed problem:
\begin{theorem}\label{t9}
Let $A\subseteq\mathfrak{X}\times\mathfrak{Z}$ and $B\subseteq\mathfrak{X}\times\mathfrak{Y}$ be two linear relations. The following statements are equivalent:
\begin{itemize}
\item[$(i)$] There exists an operator $C$ between $\mathfrak{Y}$ and $\mathfrak{Z}$ such that $A\subseteq CB$;
\item[$(ii)$] There exist a:
\begin{itemize}
\item[$(a)$] Hamel basis $\{z_\alpha\}_{\alpha\in I}$ for $\ran A$,
\item[$(b)$] linearly independent family $\{ y_\alpha\}_{\alpha\in I}$ such that $y_\alpha\in BA^{-1}(z_\alpha),\ \alpha\in I$,
\item[$(c)$] Hamel basis $\{ x'_\beta\}_{\beta\in J}$ for $\ker A,$
\item[$(d)$] family $\{ y'_\beta\}_{\beta\in J}$ such that $y'_\beta\in B(x'_\beta),\ \beta\in J$ 
\end{itemize}
and 
\begin{equation*}
\spc\{ y_\alpha\}_{\alpha\in I}\cap\spc\{ y'_\beta\}_{\beta\in J}=\{ 0\};
\end{equation*}
\item[$(iii)$]
\begin{itemize}
\item[$(a)$] For every Hamel basis $\{ z_\alpha\}_{\alpha\in I}$ for $\ran A$,
\item[$(b)$] there exists a linearly independent family $\{ y_\alpha\}_{\alpha\in I}$ such that $y_\alpha\in BA^{-1}(z_\alpha),\ \alpha\in I$ 
\end{itemize}
and
\begin{itemize}
\item[$(c)$] for every Hamel basis $\{ x'_\beta\}_{\beta\in J}$ for $\ker A,$
\item[$(d)$] there exists a family $\{ y'_\beta\}_{\beta\in J}$ such that $y'_\beta\in B(x'_\beta),\ \beta\in J$
\end{itemize}
and
\begin{equation*}
\spc\{ y_\alpha\}_{\alpha\in I}\cap \spc\{ y'_\beta\}_{\beta\in J}=\{0\}.
\end{equation*}
\end{itemize}
\end{theorem}

\begin{proof}
The implication $(iii)\Rightarrow(ii)$ is obvious.

$(i)\Rightarrow(iii).$ Let $C$ be an operator between $\mathfrak{Y}$ and $\mathfrak{Z}$ such that $A\subseteq CB$ and $\{ z_\alpha\}_{\alpha\in I}$ a Hamel basis for $\ran A.$ Then, for every fixed $\alpha\in I,$ there exists $x_\alpha\in\mathfrak{X}$ with $(x_\alpha, z_\alpha)\in A\subseteq CB.$ It follows, by the definition of the product relation, that there exists $y_\alpha\in B(x_\alpha)$ such that $z_\alpha=C(y_\alpha).$ We deduce that $y_\alpha\in BA^{-1}(z_\alpha).$ Moreover, the family $\{ y_\alpha\}_{\alpha\in I}$ is linearly independent: if $\sum_{\alpha\in I_0}\lambda_\alpha y_\alpha=0$ for a certain finite set $\{ \lambda_\alpha\}_{\alpha\in I_0\subseteq I}\subseteq\mathbb{K}$ then $\sum_{\alpha\in I_0}\lambda_\alpha z_\alpha=\sum_{\alpha\in I_0}\lambda_{\alpha}C(y_\alpha)=C(0)=0,$ hence $\lambda_\alpha=0\ (\alpha\in I_0).$

Let $\{ x'_\beta\}_{\beta\in J}$ be a Hamel basis for $\ker A.$ Then, for every fixed $\beta\in J, (x'_\beta, 0)\in A\subseteq CB.$ It follows, by the definition of the product relation, that there exists $y'_\beta\in B(x'_\beta)$ such that $C(y'_\beta)=0.$

It remains to prove that $\spc\{ y_\alpha\}_{\alpha\in I}\cap \spc\{ y'_\beta\}_{\beta\in J}=\{0\}.$ To this aim let us consider finite subsets $\{ \lambda_\alpha\}_{\alpha\in I_0\subseteq I}$ and $\{\lambda'_{\beta}\}_{\beta\in J_0\subseteq J}$ of $\mathbb{K}$ such that $\sum_{\alpha\in I_0}\lambda_\alpha y_\alpha=\sum_{\beta\in J_0}\lambda'_\beta y'_\beta.$ Then
\begin{equation*}
\sum_{\alpha\in I_0}\lambda_\alpha z_\alpha=\sum_{\alpha\in I_0}\lambda_\alpha C(y_\alpha)=\sum_{\beta\in J_0}\lambda'_\beta
C(y'_\beta)=0.
\end{equation*}
We deduce that $\lambda_\alpha=0\ (\alpha\in I_0).$ Thus $0$ is the only possible element of $\spc\{y_\alpha\}_{\alpha\in I}\cap \spc\{ y'_\beta\}_{\beta\in J}.$

$(ii)\Rightarrow(i).$ Let $\{ z_\alpha\}_{\alpha\in I}, \{y_\alpha\}_{\alpha\in I}, \{ x'_\beta\}_{\beta\in J}, \{ y'_\beta\}_{\beta\in J}$ be families with the properties in the statements of $(ii).$ We define a linear operator $C$ with 
\begin{equation*}
\dom C=\spc\{ y_\alpha\}_{\alpha\in I}\oplus \spc\{ y'_\beta\}_{\beta\in J} 
\end{equation*}
by the formula
\begin{equation}\label{eq5}
C\Bigl(\sum_{\alpha\in I_0}\lambda_\alpha y_\alpha+\sum_{\beta\in J_0}\lambda'_\beta y'_\beta\Bigr):=\sum_{\alpha\in I_0}\lambda_\alpha z_\alpha
\end{equation}
for every finite sets $\{\lambda_\alpha\}_{\alpha\in I_0\subseteq I}$, $\{\lambda'_\beta\}_{\beta\in J_0\subseteq J}\subseteq \mathbb{K}$. Since the family $\{ y_\alpha\}_{\alpha\in I}$ is linearly independent and the sum between the subspaces (of $\mathfrak{Y}$) $\spc\{ y_\alpha\}_{\alpha\in I}$ and $\spc\{ y'_\beta\}_{\beta\in J}$ is direct we deduce immediately that the definition \eqref{eq5} is correct.

As $C$ is obviously linear it remains to prove that $A\subseteq CB.$ To this aim let $(x, z)\in A$ and consider a finite set $\{ \lambda_\alpha\}_{\alpha\in I_0\subseteq I}\subseteq \mathbb{K}$ such that $z=\sum_{\alpha\in I_0}\lambda_\alpha z_\alpha.$ It follows that $x-\sum_{\alpha\in I_0}\lambda_\alpha x_\alpha\in\ker A,$ where, for $\alpha\in I, x_\alpha\in A^{-1}(z_\alpha)\cap B^{-1}(y_\alpha).$ Hence 
\begin{equation*}
x=\sum_{\alpha\in I_0}\lambda_\alpha x_\alpha +\sum_{\beta\in J_0} \lambda'_\beta x'_\beta
\end{equation*}
for a certain finite set $\{ \lambda'_\beta\}_{\beta\in J_0\subseteq J}\subseteq\mathbb{K}.$ Let $y=\sum_{\alpha\in I_0}\lambda_\alpha y_\alpha+\sum_{\beta\in J_0}\lambda'_\beta y'_\beta\in\mathfrak{Y}.$
Then
\begin{equation*}
(x, y)=\sum_{\alpha\in I_0}\lambda_\alpha(x_\alpha, y_\alpha)+\sum_{\beta\in J_0} \lambda'_\beta(x'_\beta, y'_\beta)\in B
\end{equation*}
and
\begin{equation*}
C(y)=\sum_{\alpha\in I_0}\lambda_\alpha z_\alpha=z.
\end{equation*}
We deduce that $(x, z)\in CB$, as required.
\end{proof}

Every operator solution of the equation $A\subseteq XB$ has the form \eqref{eq5} or extends an operator of the form \eqref{eq5}. More precisely it holds:
\begin{remark}\label{r10}
If the families $\{ z_\alpha\}_{\alpha\in I}, \{ y_\alpha\}_{\alpha\in I}, \{ x'_\beta\}_{\beta\in J}$ and $\{ y'_\beta\}_{\beta\in J}$ satisfy Theorem \ref{t9} $(ii)$ then the operator $C$ defined by \eqref{eq5} and all of its extensions are solutions to the problem $A\subseteq XB.$ Moreover,
\begin{equation*}
\dom C = \spc\{ y_\alpha\}_{\alpha\in I}\oplus \spc\{ y'_\beta\}_{\beta\in J},\qquad
\ran C = \ran A
\end{equation*}
and
\begin{equation*}
\ker C=\spc\{ y'_\beta\}_{\beta\in J.}
\end{equation*}

Conversely, if $C$  is an operator between $\mathfrak{Y}$ and $\mathfrak{Z}$ which satisfies $A\subseteq CB$ then there exist families $\{ z_\alpha\}_{\alpha\in I}, \{ y_\alpha\}_{\alpha\in I}, \{x'_\beta\}_{\beta\in J}$ and $\{ y'_\beta\}_{\beta\in J}$ with the properties of Theorem \ref{t9} $(ii)$ such that $\dom C \supseteq \spc\{ y_\alpha\}_{\alpha\in I}\oplus \spc\{ y'_\beta\}_{\beta\in J}$ and $C|_{\spc\{y_\alpha\}_{\alpha\in I}\oplus \spc\{ y'_\beta\}_{\beta\in J}}$ has the form \eqref{eq5}.
\qed
\end{remark}

The problem $A\subseteq XB$ has an operator solution $X=C$ which is injective if and only if there exist families $\{ z_\alpha\}_{\alpha\in I}, \{ y_\alpha\}_{\alpha\in I}, \{x'_\beta\}_{\beta\in J}$ and $\{ y'_\beta\}_{\beta\in J}$ satisfying Theorem \ref{t9} $(ii)$ and, moreover,  $y'_\beta=0,\ \beta\in J.$ Condition $0\in B(x'_\beta),\ \beta\in J$ can be rewritten as $x'_\beta\in\ker B,\ \beta\in J.$ Equivalently, $\ker A\subseteq \ker B.$ We obtain that:
\begin{corollary}\label{c11}
Let $A\subseteq \mathfrak{X}\times\mathfrak{Z}$ and $B\subseteq \mathfrak{X}\times\mathfrak{Y}$ be two linear relations. The following statements are equivalent:
\begin{itemize}
\item[$(i)$] There exists an injective operator $C$ between $\mathfrak{Y}$ and $\mathfrak{Z}$ such that $A\subseteq CB;$
\item[$(ii)$] $\ker A\subseteq \ker B$ and there exist a Hamel basis $\{ z_\alpha\}_{\alpha\in I}$ for $\ran A$ and a linearly independent family $\{y_\alpha\}_{\alpha\in I}$ such that $y_\alpha\in BA^{-1}(z_\alpha),\ \alpha\in I$;
\item[$(iii)$] $\ker A \subseteq \ker B$ and for every Hamel basis $\{z_\alpha\}_{\alpha\in I}$ for $\ran A$ there exists a linearly independent family $\{ y_\alpha\}_{\alpha\in I}$ such that $y_\alpha\in BA^{-1}(z_\alpha),\ \alpha\in I$.
\end{itemize}
\end{corollary}

In view of Remark \ref{r10} and Corollary \ref{c11}, for two given linear relations $A, B\subseteq\mathfrak{X}\times\mathfrak{Z}, A\subseteq B$ if and only if 
\begin{multline*}
\ker A\subseteq\ker B \text{ and there exists a Hamel basis } \{ z_\alpha\}_{\alpha\in I} \text{ for }\ran A\\
\text{ such that }z_\alpha\in BA^{-1}(z_\alpha),\ \alpha\in I.
\end{multline*}
The last condition takes the form: 
\begin{equation*}
\text{for every }\alpha\in I \text{ there exists }x_\alpha\in\mathfrak{X}\text{ such that }(x_\alpha, z_\alpha)\in A\cap B; 
\end{equation*}
equivalently, $z_\alpha\in \ran (A\cap B),\ \alpha\in I$.

We obtain the following result which, in fact, is equivalent to the characterization given by R. Arens (Proposition \ref{p0}):
\begin{corollary}\label{c12}
Let $A$ and $B$ be two linear relations between $\mathfrak{X}$ and $\mathfrak{Z}.$ The following statements are equivalent:
\begin{itemize}
\item[$(i)$] $A\subseteq B;$
\item[$(ii)$] $\ker A\subseteq \ker B$ and $\ran(A\cap B)=\ran A;$
\item[$(iii)$] $\mul A\subseteq \mul B$ and $\dom (A\cap B)=\dom A.$
\end{itemize}
\end{corollary}

The next example shows that the equality between two given linear relations is not ensured by the equality between their domains, ranges, kernels and multivalued parts: 
\begin{example}\label{e13}
Let $\mathfrak{X}$ be a linear space with the algebraic dimension at least $2$ and $x_1, x_2\in\mathfrak{X}$ be two linearly independent vectors. We define two operators $A, B: \spc\{x_1, x_2\}\subseteq\mathfrak{X}\to\mathfrak{X}$ as
\begin{equation*}
A=1_{\spc\{x_1, x_2\}}\quad\mbox{ and }\quad B(\lambda_1x_1+\lambda_2x_2):=\lambda_1x_2+\lambda_2x_1,\ \lambda_1, \lambda_2\in\mathbb{K}.
\end{equation*}
Then $\dom A=\dom B=\ran A=\ran B=\spc\{x_1, x_2\}$ and $\ker A=\ker B=\mul A=\mul B=\{0\}.$ However $A\not\subseteq B$ and $B\not\subseteq A.$
\qed
\end{example}

\section{The Problem $A\subseteq XB$ with Operator Solutions. Dimension}\label{s5}

We continue our discussion on the problem $A\subseteq XB$ with the goal to obtain other characterizations, related to the algebraic dimension, for the existence of an operator solution $X.$

\begin{lemma}\label{l14}
Let $R\subseteq \mathfrak{X}\times\mathfrak{Y}$ be a linear relation and $\mathfrak{X}_0$ a direct summand of $\ker R$ in $\dom R.$ We consider a Hamel basis $\{ x_\alpha\}_{\alpha\in I}$ for $\mathfrak{X}_0$ and a family $y_\alpha\in R(x_\alpha),\ \alpha\in I.$ Then:
\begin{itemize}
\item[$(a)$] The family $\{ y_\alpha\}_{\alpha\in I}$ is linearly independent;
\item[$(b)$] The following decomposition holds true
\begin{equation*}
\ran R=\mul R\oplus\mathfrak{Y}_0,
\end{equation*}
where $\mathfrak{Y}_0=\spc\{y_\alpha\}_{\alpha\in I}$.
\end{itemize}
\end{lemma}

\begin{proof}
$(a)$ Let $\{ \lambda_\alpha\}_{\alpha\in I_0\subseteq I}$ be a finite subset of $\mathbb K$ such that $\sum_{\alpha\in I_0}\lambda_\alpha y_\alpha=0.$ As
\begin{equation*}
\bigl(\sum_{\alpha\in I_0} \lambda_\alpha x_\alpha, \sum_{\alpha\in I_0}\lambda_\alpha y_\alpha\bigl)=\sum_{\alpha\in I_0}\lambda_\alpha(x_\alpha, y_\alpha)\in R
\end{equation*}
it follows that $\sum_{\alpha\in I_0}\lambda_\alpha x_\alpha\in\ker R.$ But $\ker R\cap \mathfrak{X}_0=\{0\},$ so $\sum_{\alpha\in I_0}\lambda_{\alpha}x_\alpha=0$. Since the family $\{ x_\alpha\}_{\alpha\in I}$ is linearly independent we finally deduce that $\lambda_\alpha=0,\ \alpha\in I_0.$

$(b)$ Let $\{ \lambda_\alpha\}_{\alpha\in I_0\subseteq I}$ be a finite subset of $\mathbb{K}$ such that $\sum_{\alpha\in I_0} \lambda_\alpha y_\alpha\in\mul R.$ We obtain, as above, that $\lambda_\alpha=0,\ \alpha\in I_0.$ Hence the sum between $\mul R$ and $\mathfrak{Y}_0$ is also direct.

Let  $y\in\ran R$ and $x\in\dom R$ such that $(x, y)\in R.$ Since $\dom R=\ker R\oplus \mathfrak{X}_0$ there exist $x'\in\ker R$ and a finite subset $\{ \lambda_\alpha\}_{\alpha\in I_0\subseteq I}\subseteq{K}$ such that $x=x'+\sum_{\alpha\in I_0}\lambda_\alpha x_\alpha.$ Then
\begin{equation*}
\bigl(0, y-\sum_{\alpha\in I_0} \lambda_\alpha y_\alpha\bigl)=(x, y)-(x', 0)-\sum_{\alpha\in I_0}\lambda_\alpha(x_\alpha, y_\alpha)\in R.
\end{equation*}
Consequently,
\begin{equation*}
y\in\sum_{\alpha\in I_0}\lambda_\alpha y_\alpha+\mul R\subseteq \mathfrak{Y}_0\oplus \mul R.
\end{equation*}
It follows that $\ran R=\mul R\oplus\mathfrak{Y}_0$, as required.
\end{proof}

\begin{remark}\label{r15}
$(1)$ We can replace $R$ by its inverse $R^{-1}$ to obtain a converse of the previous Lemma. 

$(2)$ Let $A\subseteq \mathfrak{X}\times\mathfrak{Z}$ and $B\subseteq \mathfrak{X}\times\mathfrak{Y}$ be linear relations satisfying $\dom A\subseteq \dom B.$ Two particular cases of Lemma \ref{l14} are important in our approach:

$(a)$ $R=B|_{\ker A}$; note that, in this case, $\dom R=\ker A$, $\ran R=B(\ker A),$ $\ker R=\ker A\cap \ker B$ and $\mul R=\mul B;$

$(b)$ $R=BA^{-1}$; note that, according to formulas \eqref{eq0} and \eqref{eq0p}, $\dom R=\ran A$, $\ran R=B(\dom A)$, $\ker R=A(\dom A\cap\ker B)$ and $\mul R=B(\ker A).$
\qed
\end{remark}

\begin{corollary}\label{c16}
Let $R\subseteq \mathfrak{X}\times\mathfrak{Y}$ be a linear relation. Then
\begin{equation*}
\codim_{\dom R}(\ker R)=\codim_{\ran R}(\mul R).
\end{equation*}
\end{corollary}

\begin{corollary}\label{c17}
Let $A\subseteq\mathfrak{X}\times\mathfrak{Z}$ and $B\subseteq\mathfrak{X}\times\mathfrak{Y}$ be (graphs of) two operators and
\begin{multline*}
\mathcal{F}:=\{\mathfrak{X}_0 : \mathfrak{X}_0 \mbox{ is a linear subspace of } \dom A\cap \dom B,\\ \dom A=\ker A\oplus \mathfrak{X}_0 \mbox{ and } \mathfrak{X}_0\cap \ker B=\{0\}\}.
\end{multline*}
The following statements are equivalent:
\begin{itemize}
\item[$(i)$] There exists an operator $C$ between $\mathfrak{Y}$ and $\mathfrak{Z}$ such that $A\subseteq CB;$
\item[$(ii)$]  There exists $\mathfrak{X}_0\in\mathcal{F}$ such that $B(\ker A)\cap B(\mathfrak{X}_0)=\{0\};$
\item[$(iii)$] $\mathcal{F}\ne\emptyset$ and for every $\mathfrak{X}_0\in\mathcal{F}$ it holds $B(\ker A)\cap B(\mathfrak{X}_0)=\{0\}.$
\end{itemize}
\end{corollary}

\begin{proof}
The implication $(iii)\Rightarrow(ii)$ is obvious.

$(i)\Rightarrow(iii).$

$(a)$ Let $\{z_\alpha\}_{\alpha\in I}$ be a Hamel basis for $\ran A$ and $\{y_\alpha\}_{\alpha\in I}$ a linearly independent family such that $y_\alpha\in BA^{-1}(z_\alpha),\ \alpha\in I$ (the existence is ensured by Theorem \ref{t9} $(ii)$). According to Lemma \ref{l14} any family $\{ x_\alpha\}_{\alpha\in I}$ such that $x_\alpha\in A^{-1}(z_\alpha)\cap B^{-1}(y_\alpha),\ \alpha\in I$ is linearly independent and $\dom A=\ker A\oplus\mathfrak{X}_0,$ where $\mathfrak{X}_0:=\spc\{x_\alpha\}_{\alpha\in I}.$ For a given finite set $\{\lambda_{\alpha}\}_{\alpha\in I_0\subseteq I}\subseteq\mathbb{K},$
\begin{equation*}
\sum_{\alpha\in I_0}\lambda_\alpha x_\alpha\in\ker B\quad\text{if and only if}\quad\sum_{\alpha\in I_0}\lambda_\alpha y_\alpha\in\mul B=\{0\}.
\end{equation*}
The family $\{y_\alpha\}_{\alpha\in I}$ is linearly independent, so $\lambda_{\alpha}=0,\ \alpha\in I_0.$ It follows that $\mathfrak{X}_0\cap\ker B=\{0\}$. Consequently $\mathfrak{X}_0\in\mathcal{F}$.

$(b)$ Let $\mathfrak{X}_0\in\mathcal{F}$ and $\{x_\alpha\}_{\alpha\in I}$ a Hamel basis for $\mathfrak{X}_0.$ For each $\alpha\in I$ we define $z_\alpha=A(x_\alpha)$ and $y_\alpha=B(x_\alpha).$ It is easy to observe that, by Lemma \ref{l14}, $\{z_\alpha\}_{\alpha\in I}$ is a Hamel basis for $\ran A.$ In addition, the family $\{y_\alpha\}_{\alpha\in I}$ is linearly independent: if $\sum_{\alpha\in I_0}\lambda_\alpha y_\alpha=0$ for a certain finite set $\{\lambda_\alpha\}_{\alpha\in I_0\subseteq I}\subseteq\mathbb{K}$ then $\sum_{\alpha\in I_0}\lambda_\alpha x_\alpha\in \ker B\cap \mathfrak{X}_0=\{0\};$ since the family $\{ x_\alpha\}_{\alpha\in I}$ is linearly independent we deduce that $\lambda_\alpha=0,\ \alpha\in I_0.$

Let $\{x'_\beta\}_{\beta\in J}$ be a Hamel basis for $\ker A$ and $y'_\beta=B(x'_\beta),\ \beta\in J.$ We remark that
\begin{equation*}
\spc\{y_\alpha\}_{\alpha\in I}\cap \spc\{y'_\beta\}_{\beta\in J}=\{0\}\quad\text{if and only if}\quad B(\mathfrak{X}_0)\cap B(\ker A)=\{0\}. 
\end{equation*}
Indeed, if $\{\lambda_\alpha\}_{\alpha\in I_0\subseteq I}$ and $\{\lambda'_\beta\}_{\beta\in J_0\subseteq J}$ are finite subsets of $\mathbb{K}$ then 
\begin{equation*}
y=\sum_{\alpha\in I_0}\lambda_\alpha y_\alpha=\sum_{\beta\in J_0}\lambda'_\beta y'_\beta\in \spc \{ y_\alpha\}_{\alpha\in I}\cap \spc\{y'_\beta\}_{\beta\in J_0}.
\end{equation*}
Equivalently,
\begin{equation*}
y=B\bigl(\sum_{\alpha\in I_0}\lambda_\alpha x_\alpha\bigr)=B\bigl(\sum_{\beta\in J_0}\lambda'_\beta x'_\beta\bigr)\in B(\mathfrak{X}_0)\cap B(\ker A).
\end{equation*}
The conclusion follows by Theorem \ref{t9} $(iii)$.

$(ii)\Rightarrow(i).$ We can proceed as in the proof of the previous implication (part $(b)$) in order to build families $\{ z_\alpha\}_{\alpha\in I}, \{ y_\alpha\}_{\alpha\in I}, \{ x'_\beta\}_{\beta\in J}$ and $\{ y'_\beta\}_{\beta\in J}$ having the properties of Theorem \ref{t9} $(ii)$. It follows, by Theorem \ref{t9} $(i)$, that there exists an operator $C$ between $\mathfrak{Y}$ and $\mathfrak{Z}$ such that $A\subseteq CB.$
\end{proof}

\begin{corollary}\label{c18}
Let $R\subseteq \mathfrak{X}\times\mathfrak{Y}$ be a linear relation. The following statements are equivalent:
\begin{itemize}
\item[$(i)$] There exist a Hamel basis $\{ x_\alpha\}_{\alpha\in I}$ for $\dom R$ and a linearly independent family $\{ y_\alpha\}_{\alpha\in I}$ with $y_\alpha\in R(x_\alpha),\ \alpha\in I;$
\item[$(ii)$] For every Hamel basis $\{x_\alpha\}_{\alpha\in I}$ for $\dom R$ there exists a linearly independent family $\{ y_\alpha\}_{\alpha\in I}$ with $y_\alpha\in R(x_\alpha),\ \alpha\in I.$
\item[$(iii)$] $\dim (\ker R)\le \dim (\mul R).$
\end{itemize}
\end{corollary}

\begin{proof}
$(i)\Rightarrow(ii).$ Let $\{ x'_\beta\}_{\beta\in J}$ and $\{ y'_\beta\}_{\beta\in J}$ be families with the properties of $(i)$. Then, for every given Hamel basis $\{ x_\alpha\}_{\alpha\in I}$ for $\dom R$ and every fixed $\alpha\in I,$ there exists a finite set $\{ \lambda_{\alpha\beta}\}_{\beta\in J(\alpha)\subseteq J}$ such that
\begin{equation*}
x_\alpha=\sum_{\beta\in J(\alpha)}\lambda_{\alpha\beta}x'_\beta.
\end{equation*}
We observe that
\begin{equation*}
(x_\alpha, y_\alpha)=\sum_{\beta\in J(\alpha)}\lambda_{\alpha\beta}(x'_\beta, y'_\beta)\in R,
\end{equation*}
where $y_\alpha:=\sum_{\beta\in J(\alpha)}\lambda_{\alpha\beta} y'_\beta.$ It remains to show that the family $\{ y_\alpha\}_{\alpha\in I}$ is linearly independent. To this aim let $\{ \lambda_\alpha\}_{\alpha\in I_0\subseteq I}\subseteq \mathbb{K}$ be a finite set with the property that 
\begin{equation}\label{eq6}
\sum_{\alpha\in I_0}\lambda_\alpha y_\alpha=\sum_{\alpha\in I_0}\sum_{\beta\in J(\alpha)} \lambda_\alpha\lambda_{\alpha\beta} y'_\beta=0.
\end{equation}
With the notations $\lambda_{\alpha\beta}=0$ for $\beta\in J\setminus J(\alpha)$ and $\alpha\in I$, the formula \eqref{eq6} can be rewritten as
\begin{equation*}
\sum_{\beta\in J}\Bigl(\sum_{\alpha\in I_0}\lambda_{\alpha}\lambda_{\alpha\beta}\Bigr)y'_\beta=0.
\end{equation*}
Equivalently, due to the fact that the family $\{ y'_\beta\}_{\beta\in J}$ is linearly independent, 
\begin{equation*}
\sum_{\alpha\in I_0}\lambda_\alpha\lambda_{\alpha\beta}=0,\quad \beta\in J.
\end{equation*}
We deduce that
\begin{eqnarray*}
0&=&\sum_{\beta\in J}\bigl(\sum_{\alpha\in I_0}\lambda_\alpha\lambda_{\alpha\beta}\bigl)x'_\beta\\
&=&\sum_{\alpha\in I_0}\lambda_\alpha\sum_{\beta\in J(\alpha)}\lambda_{\alpha\beta}x'_\beta\\
&=&\sum_{\alpha\in I_0}\lambda_\alpha x_\alpha.
\end{eqnarray*}
Since the family $\{ x_\alpha\}_{\alpha\in I}$ is linearly independent it follows that $\lambda_\alpha=0,\ \alpha\in I_0,$ as required.

$(ii)\Rightarrow(iii).$ Let $\{x_\alpha\}_{\alpha\in I_0}$ be a Hamel basis for $\ker R$ and $\{ x_\alpha\}_{\alpha\in I\supseteq I_0}$ its completion to a Hamel basis for $\dom R.$ By $(ii)$ there exists a linearly independent family $\{y_\alpha\}_{\alpha\in I}$ such that $y_\alpha\in R(x_\alpha),\ \alpha\in I.$ Note that the family $\{ y_\alpha\}_{\alpha\in I_0}$ is contained in $\mul R$. Hence $\dim(\mul R)\ge \card I_0=\dim (\ker R).$

$(iii)\Rightarrow(i).$ Let $\mathfrak{X}_0$ be a direct summand of $\ker R$ in $\dom R, \{x_\alpha\}_{\alpha\in I_0}$ a Hamel basis for $\ker R$ and $\{x_\alpha\}_{\alpha\in I\setminus I_0}$ a Hamel basis for $\mathfrak{X}_0.$ Then $\{ x_\alpha\}_{\alpha\in I}$ is a Hamel basis for $\dom R.$ We define, in view of $(iii)$, a linearly independent family $\{ y_\alpha\}_{\alpha\in I_0}$ in $\mul R.$ If, for $\alpha\in I\setminus I_0, y_\alpha\in R(x_\alpha)$ then, according to Lemma \ref{l14}, the family $\{ y_\alpha\}_{\alpha\in I\setminus I_0}$ is linearly independent and $\ran R=\mul R\oplus \spc\{y_\alpha\}_{\alpha\in I\setminus I_0}.$ It follows that $\{ y_\alpha\}_{\alpha\in I}$ is also linearly independent. The proof is complete.
\end{proof}

\begin{remark}\label{r19}
Let $A\subseteq \mathfrak{X}\times\mathfrak{Z}$ and $B\subseteq\mathfrak{X}\times\mathfrak{Y}$ be two linear relations satisfying $\dom A\subseteq \dom B.$ We can specialize Corollary \ref{c18} for the case $R=BA^{-1}:$

\textit{
The following statements are equivalent:
\begin{itemize}
\item[$(i)$] There exist a Hamel basis $\{ z_\alpha\}_{\alpha\in I}$ for $\ran A$ and a linearly independent family $\{y_\alpha\}_{\alpha\in I}$ such that $y_\alpha\in BA^{-1}(z_\alpha),\ \alpha \in I;$
\item[$(ii)$] For every Hamel basis $\{ z_\alpha\}_{\alpha\in I}$ for $\ran A$ there exists a linearly independent family $\{ y_\alpha\}_{\alpha\in I}$ such that $y_\alpha\in BA^{-1}(z_\alpha),\ \alpha\in I;$
\item[$(iii)$] $\dim[A(\dom A\cap \ker B)]\le \dim [B(\ker A)].$
\end{itemize}
}\qed
\end{remark}

Combining Corollary \ref{c11} with Remark \ref{r19} we obtain another necessary and sufficient condition for the existence of an injective operator as a solution to the problem $A\subseteq XB:$
\begin{corollary}\label{c20}
Let $A\subseteq\mathfrak{X}\times\mathfrak{Z}$ and $B\subseteq\mathfrak{X}\times\mathfrak{Y}$ be two linear relations. The following statements are equivalent:
\begin{itemize}
\item[$(i)$] There exists an injective operator $C$ between $\mathfrak{Y}$ and $\mathfrak{Z}$ such that $A\subseteq CB;$
\item[$(ii)$] 
\begin{equation*}
\dom A\subseteq \dom B,\qquad \ker A\subseteq \ker B
\end{equation*}
and 
\begin{equation*}
\dim[A(\dom A\cap \ker B)]\le\dim[B(\ker A)].
\end{equation*}
\end{itemize}
\end{corollary}

The main result of this section characterize the existence of an operator solution for the problem $A\subseteq XB$ in terms of the algebraic dimensions of the multivalued part of $B$, respectively the kernel of $BA^{-1}:$
\begin{theorem}\label{t21}
Let $A\subseteq \mathfrak{X}\times\mathfrak{Z}$ and $B\subseteq\mathfrak{X}\times\mathfrak{Y}$ be two linear relations. The following statements are equivalent:
\begin{itemize}
\item[$(i)$] There exists an operator  $C$ between $\mathfrak{Y}$ and $\mathfrak{Z}$ such that $A\subseteq CB;$
\item[$(ii)$] $\dom A\subseteq \dom B$ and $\dim(\mul B)\ge \dim[A(\dom A\cap \ker B)]$.
\end{itemize}
\end{theorem}

\begin{proof}
$(i)\Rightarrow(ii).$ Let $\mathfrak{Z}_0$ be a direct summand of $A(\dom A\cap \ker B)$ in $\ran A$ and $\mathfrak{X}_0$ a direct summand of $\ker A\cap \ker B$ in $\ker A.$ We consider a Hamel basis $\{z_\alpha\}_{\alpha\in I}$ for $\ran A$ such that $\{z_\alpha\}_{\alpha\in I_0\subseteq I}$ is a Hamel basis for $A(\dom A\cap \ker B)$ and $\{z_\alpha\}_{\alpha\in I\setminus I_0}$ is a Hamel basis for $\mathfrak{Z}_0.$ Similarly, let $\{x'_\beta\}_{\beta\in J}$ be a Hamel basis for $\ker A$ such that $\{x'_\beta\}_{\beta\in J_0\subseteq J}$ is a Hamel basis for $\ker A\cap \ker B$, while $\{ x'_\beta\}_{\beta\in J\setminus J_0}$ is a Hamel basis for $\mathfrak{X}_0.$ According to Theorem \ref{t9} (implication $(i)\Rightarrow(iii)$) there exist families $\{y_\alpha\}_{\alpha\in I}, \{y'_\beta\}_{\beta\in J}\subseteq\mathfrak{Y}$ such that $\{ y_\alpha\}_{\alpha\in I}$ is linearly independent, $y_\alpha\in BA^{-1}(z_\alpha)$ for $\alpha\in I$, $y'_\beta\in B(x'_\beta)$ for $\beta\in J$ and
\begin{equation}\label{eq7}
\spc\{y_\alpha\}_{\alpha\in I}\cap \spc\{y'_\beta\}_{\beta\in J}=\{0\}.
\end{equation} 
Let us note that, by Remark \ref{r15} $(2a)$, the family $\{y'_\beta\}_{\beta\in J\setminus J_0}$ is linearly independent and
\begin{equation}\label{eq8}
\mul B\oplus \spc\{y'_\beta\}_{\beta\in J\setminus J_0}=B(\ker A).
\end{equation}
Then, in view of \eqref{eq7}, the family 
\begin{equation*}
\mathcal{B}_0:=\{ y_\alpha,\ \alpha\in I_0;\ y'_\beta,\ \beta\in J\setminus J_0\}
\end{equation*}
is also linearly independent in $B(\ker A).$ We add the vectors $\{ y''_\gamma\}_{\gamma\in K}$ in order to complete $\mathcal{B}_0$ to a Hamel basis of $B(\ker A).$ Hence
\begin{equation}\label{eq9}
\spc\{y'_\beta\}_{\beta\in J\setminus J_0}\oplus \spc\{y_\alpha,\ \alpha\in I_0;\ y''_\gamma,\ \gamma\in K\}=B(\ker A).
\end{equation}
Following \eqref{eq8} and \eqref{eq9} we obtain that
\begin{eqnarray*}
\dim(\mul B)&=&\dim(\spc\{y_\alpha,\ \alpha\in I_0;\ y''_\gamma,\ \gamma\in K\})\\
& &\qquad\qquad\qquad\qquad(=\codim_{B(\ker A)}\spc\{y'_\beta\}_{\beta\in J\setminus J_0})\\
&=&\card I_0+\card K\\
&=&\dim [A(\dom A\cap \ker B)]+\card K\\
&\ge& \dim[A(\dom A\cap \ker B)].
\end{eqnarray*}

$(ii)\Rightarrow(i).$ Let $\mathfrak{Z}_0$ and $\mathfrak{X}_0$ be linear subspaces of $\ran A$ and, respectively $\ker A$ such that
\begin{equation*}
A(\dom A\cap \ker B)\oplus \mathfrak{Z}_0=\ran A\quad\mbox{and}\quad
(\ker A\cap \ker B)\oplus \mathfrak{X}_0=\ker A.
\end{equation*}

Consider also the Hamel bases $\{ z_\alpha\}_{\alpha\in I_0}$ for $A(\dom A\cap \ker B)$, $\{z_\alpha\}_{\alpha\in I\setminus I_0}$ for $\mathfrak{Z}_0$, $\{x'_\beta\}_{\beta\in J_0}$ for $\ker A\cap \ker B$ and $\{x'_\beta\}_{\beta\in J\setminus J_0}$ for $\mathfrak{X}_0.$ As $\dim(\mul B)\ge\dim [A(\dom A\cap \ker B)]$ there exists a linearly independent family $\{y_\alpha\}_{\alpha\in I_0}\subseteq \mul B.$ For each $\alpha\in I\setminus I_0$ let $y_\alpha\in BA^{-1}(z_\alpha).$ By Remark \ref{r15} $(2b)$ the family $\{y_\alpha\}_{\alpha\in I\setminus I_0}$ is linearly independent and
\begin{equation*}
B(\ker A)\oplus \mathfrak{Y}_0=B(\dom A),
\end{equation*}
where $\mathfrak{Y}_0=\spc\{y_\alpha\}_{\alpha\in I\setminus I_0}.$ Since $\mul B\subseteq B(\ker A)$ and $B(\ker A)\cap \mathfrak{Y}_0=\{0\}$ it follows that the family $\{ y_\alpha\}_{\alpha\in I}$ is also linearly independent. Moreover, $y_\alpha\in BA^{-1}(z_\alpha)$ for every $\alpha\in I.$ Let $y'_\beta\in B(x'_\beta),\ \beta\in J$ such that, for $\beta\in J_0, y'_\beta=0.$

We claim that
\begin{equation*}
\spc\{y_\alpha\}_{\alpha\in I}\cap \spc\{y'_\beta\}_{\beta\in J}=\{0\}.
\end{equation*}
To this aim, let $\{\lambda_\alpha\}_{\alpha\in I_1\subseteq I}$ and $\{ \lambda'_\beta\}_{\beta\in J_1\subseteq J\setminus J_0}$ be finite subsets of $\mathbb{K}$ such that 
\begin{equation*}
\sum_{\alpha\in I_0\cap I_1}\lambda_\alpha y_\alpha+\sum_{\alpha\in I_1\setminus I_0}\lambda_\alpha y_\alpha=\sum_{\beta\in J_1}\lambda'_\beta y'_\beta.
\end{equation*}
We deduce that
\begin{equation*}
\sum_{\alpha\in I_1\setminus I_0}\lambda_\alpha y_\alpha\in B(\ker A)\cap{\mathfrak{Y}_0}=\{0\},
\end{equation*}
so $\lambda_\alpha=0,\ \alpha\in I_1\setminus I_0.$ Consequently,
\begin{equation*}
\sum_{\beta\in J_1}\lambda'_\beta y'_\beta=\sum_{\alpha\in I_0\cap I_1}\lambda_\alpha y_\alpha\in \mul B.
\end{equation*}
Equivalently,
\begin{equation*}
\sum_{\beta\in J_1}\lambda'_\beta x'_\beta\in\ker B\cap\ker A\cap\mathfrak{X}_0=\{0\}.
\end{equation*}
Thus, $\lambda'_\beta=0,\ \beta\in J$ which proves our claim. $(i)$ follows immediately by Theorem \ref{t9} (implication $(ii)\Rightarrow(i)$).
\end{proof}

Condition $(ii)$ of Theorem \ref{t21} is fulfilled if, in particular, $A(\dom A\cap \ker B)=\{ 0\}$ or, in equivalent form, $A$ is (the graph of) an operator and $\dom A\cap \ker B\subseteq \ker A.$ Indeed, if $x\in\ker B\cap\dom A$ then, for a certain $z\in\mathfrak{Z}, (x, z)\in A.$ It follows that $z\in A (\dom A\cap \ker B)=\{ 0\},$ so $x\in \ker A.$ Moreover, 
\begin{equation*}
\mul A\subseteq A(\dom A\cap \ker B)=\{0\}.
\end{equation*}
Conversely, let $z\in A(\dom A\cap \ker B)$ and $x\in \ker B\cap \dom A\subseteq \ker A$ such that $(x, z)\in A.$ It follows that $z\in \mul A=\{ 0 \},$ that is $A(\dom A\cap \ker B)=\{ 0 \}.$ We deduce that:
\begin{corollary}\label{c22}
Let $A\subseteq\mathfrak{X}\times\mathfrak{Z}$ be (the graph of) an operator and $B\subseteq\mathfrak{X}\times\mathfrak{Y}$ a linear relation such that
\begin{equation*}
\dom A\subseteq \dom B\quad\text{and}\quad\ker B\cap \dom A\subseteq \ker A. 
\end{equation*}
Then there exists an operator $C$ between $\mathfrak{Y}$ and $\mathfrak{Z}$ such that $A\subseteq CB.$
\end{corollary}

\begin{remark}\label{r23}
If $A$ and $B$ are both (graphs of) operators then conditions $\dom A\subseteq\dom B$ and $\ker A\cap\dom A\subseteq \ker A$ are also sufficient for the existence of an operator $C$ between $\mathfrak{Y}$ and $\mathfrak{Z}$ such that $A\subseteq CB.$
\qed
\end{remark}

\begin{corollary}\label{c24}
Let $R\subseteq\mathfrak{X}\times\mathfrak{Y}$ be a linear relation and 
\begin{equation*}
\Delta_{\ran R}:=\{ (y, y) : y\in \ran R\}\subseteq \mathfrak{Y}\times\mathfrak{Y}.
\end{equation*}
The following statements are equivalent:
\begin{itemize}
\item[$(i)$] There exists an operator $C$ between $\mathfrak{X}$ and $\mathfrak{Y}$ such that $\Delta_{\ran R}\subseteq CR^{-1};$
\item[$(ii)$] There exists an operator $C\subseteq R$ such that $\ran C=\ran R;$
\item[$(iii)$] $\dim (\ker R)\ge \dim (\mul R).$
\end{itemize}
\end{corollary}

\begin{proof}
The equivalence $(i)\Leftrightarrow(ii)$ follows easily by a standard approach, while $(i)\Leftrightarrow(iii)$ by Theorem \ref{t21} for $A=\Delta_{\ran R}$ and $B=R^{-1}.$
\end{proof}

\begin{remark}\label{r25}
We can specialize Corollary \ref{c24} for the relations $R^{-1},$ respectively $AB^{-1}$ to obtain new necessary and sufficient conditions in Corollary \ref{c18}, respectively Remark \ref{r19} and Corollary \ref{c20}.
\qed
\end{remark}

\begin{acknowledgements}
The work of the first author has been accomplished during his visit (supported by the Hungarian Scholarship Board) at the Department of Applied Analysis and Computational Mathematics, E\"otv\"os Lor\'and University of Budapest. He will take this opportunity to express his gratitude to all the members of this department, especially to Professor Zolt\'an Sebesty\'en, for their kind hospitality.
\end{acknowledgements}








\end{document}